\documentclass[12pt, twoside, leqno]{amsart}

\usepackage{amssymb}

\newtheorem{theorem}{Theorem}[section]
\newtheorem{corollary}[theorem]{Corollary}

\theoremstyle{definition}
\newtheorem{definition}[theorem]{Definition}
\newtheorem{remark}[theorem]{Remark}
\newtheorem{example}[theorem]{Example}
\numberwithin{equation}{section}

\newcommand{\C}{\mathbb{C}}

\newcommand{\be}{\begin{equation}}
\newcommand{\ee}{\end{equation}}
\newcommand{\bp}{\begin{proof}}
\newcommand{\ep}{\end{proof}}
\newcommand{\ben}{\begin{enumerate}}
\newcommand{\een}{\end{enumerate}}
\newcommand{\ba}{\begin{array}}
\newcommand{\ea}{\end{array}}
\newcommand{\bal}{\begin{align}}
\newcommand{\eal}{\end{align}}

\def\Real{\mathop{\rm Re}}

\def\rc#1{\frac{1}{#1}}

\begin{document}
\title{On Divisor Function Polynomials and Dirichlet Series}
%%% Give just the affiliation and not the detailed postal address.
\author{M. N. Huxley, M. C. Lettington, K. M. Schmidt\\
Cardiff University UK} 
\date{October 9, 2026}
\maketitle
%% Classification and key words; note that the 2020 classification is used.
%% Only one classification should be declared as primary.
\renewcommand{\thefootnote}{}



\renewcommand{\thefootnote}{\arabic{footnote}}
\setcounter{footnote}{0}

%%%%%%%%

\begin{abstract}
We find the Dirichlet series and convergence bounds for the counting function of sum systems
(collections of sets of non-negative integers whose Minkowski sum is an
arithmetic progression).
This motivates the definition and study of divisor function polynomials,
which are generating functions for the non-trivial divisor function and
for the balanced divisor function that counts, up to a sign, non-trivial
ordered factorisations into squarefree factors. These divisor function polynomials
turn out to be intimately connected between each other and with Bernoulli
numbers and ordered Bell polynomials.
\end{abstract}
\section{Introduction}
\markboth{M. N. Huxley, M. C. Lettington and K. M. Schmidt}{On Divisor Function Polynomials and Dirichlet Series}
An interesting variant of the classical $j$-th divisor function $d_j$,
which counts ordered factorisations of natural numbers into $j$ factors,
and of the somewhat less studied non-trivial divisor function $c_j$ (see \cite{ect} pp. 1-7),
which counts ordered factorisations with all factors $> 1$,
appears in the counting formula for $m$-part sum systems
%In what follows we are primarily concerned with arithmetic functions, their Dirichlet series and Taylor polynomials, related to the balanced divisor function $c_j^{(-j)}(n)$, defined for all natural number $n$, with $j\geq 0$.
%The value of the balanced divisor function has the combinatorial interpretation of
%$(-1)^{\Omega(n) + j}$ times the number of ordered factorisations of $n$ into $j$ non-trivial, squarefree factors (see \S 2 of \cite{mcl4}), where $\Omega(n)$ counts the number of prime factors of $n$ including repeats. We notate its absolute value as $|c_j^{(-j)}|=b_j$.
%
%The balanced divisor function appears in the counting formula
\cite[Theorem 1.5]{adl},
%for the number of all $m$-part joint ordered factorisations arising from a given natural number $n$, denoted by $\mathcal{N}_m(n)$.
%[EXPAND A BIT:]
%These are ordered factorisations of an $n$-tuple of natural numbers $a =(n_1, . ..,n_m)$, arising from writing each of these numbers as a product of non-trivial factors, i.e., factors $\geq 2$, and then arranging all factors in a linear chain such that no two adjacent factors arise from the factorisation of the same number $n_j$.
%
%For natural numbers $m,n \in \mathbb{N}$, the number of $m$-part joint ordered factorisations arising from $n$
%is equal to
\be\label{enum22}
 {\mathcal N}_m(n)
 %= \sum_{L=0}^{\Omega{(n)}} (-1)^L m!\, S(L,m)\, (\mu-e)^{*L}(n)
 = \sum_{L=1}^{\Omega{(n)}} m!\, S(L,m)\, c_L^{(-L)}(n).
\ee
An $m$-part sum system is a collection of subsets $A_1, \dots, A_m$ of
$\mathbb N_0$ such that the Minkowski sum
\begin{equation*}
A_1 + \cdots + A_m := \left\{\sum_{j=1}^m a_j \mid a_j\in A_j\ (j\in\{1, \dots, m\})\right\} = \{0, 1, 2, \dots, n-1\},
\end{equation*}
where each number in the target set appears exactly once as a sum
(or equivalently, the product of the cardinalities of $A_1, \dots, A_m$ equals $n$).
$\mathcal N_m(n)$ is the number of all different $m$-part sum systems for the
given $n$.
In formula (\ref{enum22}), we see that alongside the well-known Stirling numbers of the second kind, $S(L, m)$, which can be defined as the number of ways
to partition a set of $L$ objects into $m$ non-empty subsets
\cite{riordan}, there appear the arithmetic functions $c_j^{(-j)}$, which we
call the balanced divisor functions.
The number $c_j^{(-j)}(n)$ counts, up to the factor $(-1)^{j + \Omega(n)}$,
the non-trivial ordered factorisations of $n$ into squarefree factors \cite{mcl4}.

In the present paper, we investigate the arithmetic functions $b_j = |c_j^{(-j)}|$ further (Section \ref{sec:df}) and observe that $c_j^{(-j)}$ and $c_j$ are linked in a surprising
way as their generating functions (divisor function polynomials) essentially
coincide up to a unitary shift.
We find the Dirichlet series for these divisor function polynomials (in Section \ref{sec:dfp}) and of the sum system counting function $\mathcal N_m$ (in Section \ref{sec:DN}).
In Section \ref{sec:Bell} we explore connections between the divisor function
polynomials for squarefree numbers and the Bernoulli numbers, ordered Bell numbers, ordered Bell polynomials
and a signed variant of the latter.
%[I CHANGED THE SUM FROM $L\ge 0$ TO $L\ge 1$ TO AVOID LATER COMPLICATIONS;
%THIS SHOULD GIVE THE SAME AS $m\ge 1$.]
%where the well-known Stirling numbers of the second kind $S(L,m)$ count the number of ways to partition a set of $L$ objects into $m$ nonempty subsets \cite{riordan}.
%
%
%In contrast [TO WHAT? PLACE?], the classical $j$-th divisor function $d_j(n)$ counts the number of ordered factorisations
%$n=n_1 n_2\ldots n_j$, with each factor $n_k\geq 1$. When we require that each factor is non-trivial, so each $n_k\geq 2$, we have the less ubiquitous $j$-th non-trivial divisor function, $c_j(n)$ (cf. \S 1 of \cite{ect} or \cite{mcl2}).
%
%Three other arithmetic functions which naturally arise in our considerations are the classical M\"obius function $\mu$ defined by $\mu(n)=(-1)^{\Omega(n)}$ if $n$ is squarefree (has no repeated prime factors) and $0$ otherwise; Liouville's function $\lambda$ defined for all natural numbers $n$ by $\lambda(n)=(-1)^{\Omega(n)}$, and the Dirichlet convolution identity function $e$, defined by $e(n)=\delta_{n,1}$, which takes the value $1$ when $n=1$ and 0 otherwise. These arithmetic functions obey the pointwise product relations
%\[
%(-1)^j\lambda b_j=c_j^{(-j)},\quad \lambda\mu=|\mu|,\quad \lambda e =\mu e =e;
%\]
%where we set $c_0^{(0)}=d_0=c_0=e$ [PLACE?].

\section{Divisor functions}\label{sec:df}

Divisor functions can conveniently be expressed in terms of the commutative Dirichlet convolution algebra of arithmetic functions, defined below.
\begin{definition} The {\it Dirichlet convolution\/} of arithmetic functions $f_1, f_2, \dots, f_j$
is defined as
\begin{align}
 (f_1 * f_2 * \cdots * f_j)(n) &= \sum_{n_1 n_2 \cdots n_j = n} f_1(n_1) f_2(n_2) \cdots f_j(n_j),
\label{econvo}
\end{align}
summing over all ordered factorisations of $n \in {\mathbb N}$ into $j$ factors.
The convolution is evidently commutative and linear in each function $f_1, \dots, f_j$.
The $j$th convolution power is notated as $f^{*j} := f * f * \cdots * f$,
where $f$ is repeated $j$ times on the right-hand side.
\end{definition}
The function $e$, defined by $e(n) = \delta_{n, 1}$, which takes the value 1
when $n = 1$ and~0 otherwise, is the neutral element of the Dirichlet convolution
product; hence $f^{*0} = e$.
We write $f^{-*}$ for the convolution inverse (which exists if and only if
$f(1) \neq 0$ \cite[Theorem 2.2.(2)]{rSchSp}). The convolution inverse of the constant function 1
is the M\"obius function,
$1^{-*} = \mu$, which is defined as $\mu(n) = (-1)^{\Omega(n)}$ if $n$ is
squarefree (has no repeated prime factors) and 0 otherwise.

The divisor function $d_j$ counting the number of ordered factorisations can
be written as $d_j = 1^{*j}$; the non-trivial divisor function $c_j$ counting
the number of ordered factorisations into factors $> 1$ can be written as
$c_j = (1 - e)^{*j}$. Hence the sum-over-divisors relations, in more explicit
notation
\[
d_j(n)=\sum_{k\mid n}d_{j-1}(k),\qquad c_j(n)=\sum_{k\mid n,\ k<n}c_{j-1}(k)\qquad (n\in\mathbb N; j\in\mathbb N),
\]
follow.

To complement $c_j$, associated divisor functions $c_j^{(r)}$ (with $j \in
\mathbb N_0$ and $r\in\mathbb Z$) were introduced in
\cite{mcl2} and \cite{mcl4}. For $r > 0$, $c_j^{(r)} = (1 - e)^{*j}*1^{*r}
= c_j * d_r$
is a hybrid of the plain and non-trivial divisor functions; $c_j^{(0)} = c_j$.
For negative upper index, we have $c_j^{(-r)} = (1 - e)^{*j} * \mu^{*r}$ (where $r > 0$), using the fact that the M\"obius function is the convolution inverse of 1.
The functions $\mu^{*r} = c_0^{(-r)}$ have been studied by Popovici \cite{rPopo}.

The associated divisor function with upper index equal to the negative of the lower
index, $c_j^{(-j)} = (e - \mu)^{*j} = (-1)^j\,(\mu - e)^{*j}$, which we here call the {\it balanced divisor function\/},
plays a special role as it appears naturally in the counting of sum systems
\cite{mcl2, mcl4, adl}, see equation (\ref{enum22}).
The value $c_j^{(-j)}(n)$ can be interpreted as counting, up to the sign
factor $(-1)^{j + \Omega(n)}$, the number of ordered factorisations of $n$
into $j$ squarefree factors $> 1$ (see \cite{mcl4}). This motivates the definition of the
{\it squarefree non-trivial divisor function\/}
\begin{equation}\label{eq:bdef}
b_j := |c_j^{(-j)}| = (-1)^j\,\lambda\,c_j^{(-j)},
\end{equation}
with Liouville's function
$\lambda(n) = (-1)^{\Omega(n)}$ $(n\in\mathbb N)$.
We immediately obtain the relations
\begin{equation*}
 c_j =(1-e)^{*j}= 1^{*j}*(e - \mu)^{*j} = d_j * c_j^{(-j)}
\end{equation*}
and
\be
\label{lembj}
 b_j = (-1)^j \lambda\,c_j^{(-j)} = (-1)^j \lambda\,(e - \mu)^{*j} = \lambda\,(\mu - e)^{*j}.
\ee
The latter equation can be simplified as shown in the following theorem, where $\zeta(s)$ denotes the Riemann zeta function \cite{ivic,mcl5}.
\begin{theorem}\label{thm:Db}
For $j\in\mathbb N_0$,
$$ b_j = (|\mu| - e)^{*j}. $$
The corresponding Dirichlet series has the form
\[
\sum_{n=1}^\infty \frac{b_j}{n^s}=\left ( \frac{\zeta(s)}{\zeta(2s)}-1\right)^j
\qquad (\Real s > 1).
\]
\end{theorem}
\begin{proof} From equation (\ref{lembj})
\begin{align*}
 b_j(n) &= \lambda(n) \sum_{n_1 n_2 \cdots n_j = n} \prod_{k=1}^j (\mu(n_k) - e(n_k))
 = \sum_{n_1 n_2 \cdots n_j = n} \prod_{k=1}^j \lambda(n_k) (\mu(n_k) - e(n_k))
 & \\
 &= \sum_{n_1 n_2 \cdots n_j = n} \prod_{k=1}^j (|\mu(n_k)| - e(n_k))
 = (|\mu| - e)^{*j}(n)
\qquad (n\in\mathbb N),
 \end{align*}
where we used the complete multiplicativity of
$\lambda$
and the facts that
$\lambda \mu = |\mu|$
and
$\lambda e = e.$

For the Dirichlet series, we first note that for a general arithmetic function $f(n)$ with Dirichlet series $F(s)$, the arithmetic function $f(n)-e(n)$ has Dirichlet series
\[
\sum_{n=1}^\infty \frac{f(n)-e(n)}{n^s}=\sum_{n=1}^\infty \frac{f(n)}{n^s}-1=F(s)-1.
\]
Also, the Dirichlet series of a convolution of arithmetic functions is equal
to the product of their individual Dirichlet series.
Hence the $j$-th convolution power $(f(n)-e(n))^{*j}$ has Dirichlet series $(F(s)-1)^j$.

The Dirichlet series for Liouville's function $\lambda$ is given by $\zeta(2s)/\zeta(s)$ \cite{Lehman}.
Moreover, the Dirichlet convolution inverse of $\lambda$ is
the absolute value of the Möbius function \cite[Ex.\ I.8.(8)]{rSchSp},
$\lambda^{-*}(n) = \mu(n)\,\lambda(n) = |\mu(n)| = \mu^2(n)$, which is the characteristic function of the squarefree integers.

It follows that the function $(|\mu| - e)$ has the Dirichlet series $\zeta(s)/\zeta(2s)-1$ (see also \cite[Corollary 8.20]{NZM}), and consequently
\[
\sum_{n=1}^\infty \frac{b_j(n)}{n^s}=\left ( \frac{\zeta(s)}{\zeta(2s)}-1\right)^j.
\qedhere
\]
\end{proof}
\begin{remark}
In view of equation (\ref{eq:bdef}), %$b_j = (-1)^j \lambda\,c_j^{(-j)}=|c_j^{(-j)}|$,
it is of interest to compare this with the
Dirichlet series for $c_j^{(-j)}$
(see \cite{mcl4}),
\begin{equation}\label{eq:Dbal}
\sum_{n=1}^\infty \frac{c_j^{(-j)}(n)}{n^s} = \left (1-\frac{1}{\zeta(s)}\right )^j.
\end{equation}
\end{remark}

\section{Divisor function polynomials}\label{sec:dfp}

In the following, we consider formal power series (generating functions) for
the non-trivial and the balanced divisor functions.
In that these are sums over the index $j$ in $c_j(n)$ and $c_j^{(-j)}(n)$ for fixed integer $n\in\mathbb{N}$, they are a distant relative of the divisor summatory function.
Note that there is no question of convergence as, for each fixed $n\in\mathbb N$, only finitely many terms of the series are non-zero.
We therefore call these formal power series divisor function polynomials.
\begin{definition}\label{defn:1}
The {\it non-trivial and balanced divisor function polynomials\/} are defined for $n\in\mathbb N$ and $x\in\mathbb C$ as
\[
P(n,x)=\sum_{j=1}^\infty x^{j-1}c_j(n),\qquad Q(n,x)=\sum_{j=1}^\infty x^{j-1}c_j^{(-j)}(n),
\]
respectively.
\end{definition}
\begin{remark}
As an immediate consequence, the non-trivial and balanced divisor function polynomials are the generating functions for the respective divisor functions, satisfying
\[
\frac{1}{(k-1)!}\,\frac{\partial^{k-1}}{\partial x^{k-1}}P(n,0) =c_k(n),\quad
\frac{1}{(k-1)!}\,\frac{\partial^{k-1}}{\partial x^{k-1}}Q(n,0) =c_k^{(-k)}(n).
\]
\end{remark}
The divisor function polynomials for the non-trivial and for the balanced
divisor functions are connected and indeed arise from the same family of
polynomials, as shown in the next theorem.
\begin{theorem}
\label{thm11}
For each $n\in\mathbb N$,
the divisor function polynomials $P(n,\cdot)$ and $Q(n,\cdot)$ satisfy the identity
\be\label{PQlem}
P\left(n,x-1\right)=Q\left(n,x\right) \quad (x \in \C).
\ee
\end{theorem}

\begin{proof}
Applying Newton's binomial theorem
\be
 \rc{(1+x)^s} = \sum_{l=0}^\infty \binom{l + s - 1}{l} (-1)^l  x^l
\ee
in the commutative Dirichlet convolution algebra of arithmetic functions, we obtain the identity
\begin{align*}
 \mu^{*L} &= 1^{*-L}
 = (e + (1 - e))^{*-L}
 = \sum_{l=0}^\infty \binom{l + L - 1}{l} (-1)^l  (1 - e)^{*l}
\end{align*}
and consequently
\begin{align*}
 c_L^{(-L)} &= (e - \mu)^{*L} * 1^{*L} * \mu^{*L}\\
 &= (1 - e)^{*L} * \sum_{l=0}^\infty \binom{l + L - 1}{l} (-1)^l  (1 - e)^{*l}\\
 &= \sum_{l=0}^\infty \binom{l + L - 1}{l} (-1)^l  (1 - e)^{*l+L}
 = \sum_{l=0}^\infty \binom{l + L - 1}{l} (-1)^l  c_{l + L}.
\end{align*}
Therefore we find
\begin{align*}
 \sum_{L=1}^\infty (x - 1)^{L-1}  c_L &= \sum_{j=0}^\infty (x - 1)^j  c_{j+1}
 = \sum_{j=0}^\infty \sum_{k=0}^j \binom{j}{k} x^k (-1)^{j-k}  c_{j+1}\\
 &= \sum_{k=0}^\infty x^k \sum_{l=0}^\infty \binom{l + k}{k} (-1)^l c_{l+k+1}
 \\
 &= \sum_{L=1}^\infty x^{L-1} \sum_{l=0}^\infty \binom{l + L - 1}{l} (-1)^l  c_{l+L}
 = \sum_{L=1}^\infty x^{L-1}  c_L^{(-L)}.
\end{align*}
(Note that all sums of associated divisor functions are pointwise finite, so
there is no issue of convergence in the above calculations.)
\end{proof}

\begin{theorem}
\label{thm12}
The non-trivial divisor function polynomial $P$ has, for $\Real s > 1$ %such that $\zeta(\Real s)\neq 1$
and $x\in\mathbb C$ such that $|x| < |\zeta(\Real s) - 1|^{-1}$, the Dirichlet series
\begin{equation} \label{eq:DP}
D_{P(\cdot, x)}(s) = \sum_{n=1}^\infty \frac{P(n, x)}{n^s} = \frac{\zeta(s) - 1}{1 + x - x\,\zeta(s)};
\end{equation}
the rational function on the right-hand side gives its analytic continuation
to all $(s, x) \in\mathbb C^2$ such that $s\neq 1$ and $x\neq (\zeta(s) - 1)^{-1}$.

The analytically continued Dirichlet series for the balanced divisor function
polynomial $Q$ and the shifted polynomials $P + e$ and $Q + e$ take,
for $s\in\mathbb C\setminus\{1\}$,
 the form
\begin{align}
&D_{Q(\cdot, x)}(s) = \frac{\zeta(s) - 1}{x + (1 - x)\,\zeta(s)} \qquad (x\in\mathbb C\setminus
\{{\textstyle \frac{\zeta(s)}{\zeta(s) - 1}}\}),
\label{eq:DQ}
\\
&D_{P(\cdot, x) + e}(s) = D_{P(\cdot, x)} + 1 = \frac{x + (1 - x)\,\zeta(s)}{1 + x - x\,\zeta(s)} \qquad (x\in\mathbb C\setminus
\{{\textstyle \frac 1{\zeta(s) - 1}}\}),
\label{eq:DP1}
\\
&D_{Q(\cdot, x) + e}(s) = D_{Q(\cdot, x)}(s) + 1 = \frac{x - 1 + (2 - x)\,\zeta(s)}{x + (1 - x)\,\zeta(s)} \qquad (x\in\mathbb C\setminus
\{{\textstyle \frac{\zeta(s)}{\zeta(s) - 1}}\}).
\label{eq:DQ1}
\end{align}
Moreover, the following identities hold,
\begin{align}
 &D_{P(\cdot, x)} = D_{P(\cdot, x)+e}\,D_{Q(\cdot, x)};
 \label{eq:DPQ}
\\
 &D_{P(\cdot, x) + e}\,D_{Q(\cdot, x) + e} = 2\,D_{P(\cdot, x)} + 1.
 \nonumber
\end{align}
\end{theorem}
\begin{proof}
For $N\in\mathbb N$, let $\hat\Omega(N) := \max\limits_{1\le n\le N} \Omega(n)
= \lfloor \log_2 N \rfloor$.
For $x > 0$, we have
\begin{equation*}
 \sum_{n=1}^N \frac{P(n, x)}{n^s} = \sum_{n=1}^N \sum_{k=1}^{\hat\Omega(N)} x^{k-1} \frac{c_k(n)}{n^s}
 = \sum_{k=1}^{\hat\Omega(N)} x^{k-1} \sum_{n=1}^N \frac{c_k(n)}{n^s}
 = \sum_{k=1}^{\infty} x^{k-1} \sum_{n=1}^N \frac{c_k(n)}{n^s}.
\end{equation*}
Here
\begin{align*}
\left| x^{k-1} \sum_{n=1}^N \frac{c_k(n)}{n^s} \right|
&= |x|^{k-1}\left|\sum_{n=1}^N \frac{c_k(n)}{n^s} \right|
\le|x|^{k-1}\sum_{n=1}^N \frac{c_k(n)}{|n^s|}
\\
&= |x|^{k-1} \sum_{n=1}^N \frac{c_k(n)}{n^{\Real s}}
\le  |x|^{k-1}\,|\zeta(\Real s) - 1|^k,
\end{align*}
which is independent of $N$, and
$$
 \sum_{k=1}^\infty |x|^{k-1}\,|\zeta(\Real s) - 1|^k
= |\zeta(\Real s) - 1|\sum_{k=1}^\infty (|x|\,|\zeta(\Real s) - 1|)^{k-1} < \infty
$$
for $|x| < |\zeta(\Real s) - 1|^{-1}$, and also
\begin{equation*}
\lim_{N \rightarrow\infty} x^{k-1} \sum_{n=1}^N \frac{c_k(n)}{n^s}
= x^{k-1} (\zeta(s) - 1)^k
\end{equation*}
for each $k\in\mathbb N$,
so by the dominated convergence theorem
\begin{equation*}
D_{P(\cdot, x)}(s) = \lim_{N\rightarrow\infty} \sum_{n=1}^N \frac{P(n, x)}{n^s}
= \sum_{k=1}^\infty x^{k-1}\,(\zeta(s) - 1)^k
= \frac {\zeta(s) - 1}{1 - x\,(\zeta(s) - 1)},
\end{equation*}
giving (\ref{eq:DP}).
Hence (\ref{eq:DP1}) follows by a simple calculation. Equations (\ref{eq:DQ})
and (\ref{eq:DQ1}) arise by replacing $x$ in the formulae for $D_P$ by $x - 1$
(see Theorem \ref{thm11}).
The final two identities are a direct consequence of the formulae
(\ref{eq:DP}) -- (\ref{eq:DQ1}).
\end{proof}

\begin{theorem}
\label{thm112}
The arithmetic functions $P(\cdot,x)$ and $Q(\cdot,x)$, where $x \in \C$, obey the sum-over-divisor relations
\be\label{115}
P(n,x)=x\sum_{d\mid n,\ d<n}\left (P(d,x)+\frac{e(d)}{x}\right),
\ee
\be\label{116}
Q(n,x)=(x-1)\sum_{d\mid n,\ d<n}\left (Q(d,x)+\frac{e(d)}{x-1}\right),
\ee
and
\be\label{117}
P(n,x)=\sum_{d\mid n,\ d<n}\left (P(d,x)+e(d)\right)Q\left(\frac{n}{d},x\right).
\ee
\end{theorem}
\begin{proof}
By the definition of Dirichlet convolution, we see that for every arithmetic
function $f$,
\begin{equation*}
\sum_{d | n,\ d < n} f(d) = (f * (1 - e))(n) \qquad (n\in\mathbb N).
\end{equation*}
Hence the right-hand side of equation (\ref{115}), abstracting from the variable
$n$, is equal to
\begin{align*}
 (x\, P(\cdot, x) + e) * (1 - e) &= x\,P(\cdot, x) * (1 - e) + e * (1 - e)
\\
 &= x \sum_{j=1}^\infty x^{j-1}\,c_j * (1 - e) + 1 - e
\\
 &= \sum_{j=1}^\infty x^j\,(1 - e)^{*(j+1)} + 1 - e
 = \sum_{k=2}^\infty x^{k-1}\,(1 - e)^{*k} + 1 - e
\\
 &= \sum_{k=1}^\infty x^{k-1}\,(1 - e)^{*k}
 = P(\cdot, x).
\end{align*}
Equation (\ref{116}) now follows by Theorem \ref{thm11} as
\begin{align*}
 Q(\cdot, x) &= P(\cdot, x - 1)
 = ((x - 1) P(\cdot, x - 1) + e) * (1 - e)
\\
 &= ((x - 1) Q(\cdot, x) + e) * (1 - e).
\end{align*}
Finally, as the Dirichlet series of a convolution of arithmetic functions
is the product of their separate Dirichlet series, we
deduce from equation (\ref{eq:DPQ}) that
\begin{align*}
P(n, x) &= \left((P(\cdot, x) + e) * Q(\cdot, x)\right)(n)
= \sum_{d | n} (P(d, x) + e(d))\,Q\left(\frac n d, x\right)
\end{align*}
for all $n\in\mathbb N$, giving equation (\ref{117}) since
$Q(1, x) = 0$.
\end{proof}

\begin{remark}\label{corcom}
The polynomial sequence $(|P(n,-2)|)_{n\in\mathbb N}$ gives the number of all ordered factorisations of $n$ into squarefree numbers greater than 1 (cf.\ OEIS A050328);
more precisely, for $n\in\mathbb N$,
\begin{align*}
 P(n, -2) &= Q(n, -1)
 = \sum_{j=1}^\infty (-1)^{j-1} c_j^{(-j)}(n)
 = \sum_{j=1}^\infty (-1)^{j-1} (-1)^j \lambda(n)\,b_j(n)
\\
 &= (-1)^{\Omega(n)+1} \sum_{j=1}^\infty b_j(n).
\end{align*}
Taking $x = -1$, we find
\begin{equation*}
P(n, -1) = Q(n, 0) = c_1^{(-1)}(n) = (e - \mu)(n),
\end{equation*}
so, except at $n = 1$, its values are the negative of the M\"obius function
(cf.\ the sequence OEIS A008683).
The sequence $(P(n,0))_{n\in\mathbb N}$ is the sequence with every term (except for the leading 0) equal to 1 (cf.\ OEIS A000012),
since $P(n, 0) = c_1(n) = (1 - e)(n)$. Interestingly, by Theorem \ref{thm11}
this implies that
\begin{equation*}
 Q(n, 1) = \sum_{j=1}^\infty c_j^{(-j)}(n) = 1 \qquad (n \in \mathbb N, n \ge 2)
\end{equation*}
(cf.\ \cite[Eq. (7) with $r = 1$]{mcl4}).
Similarly, the sequence $(P(n,1))_{n \in\mathbb N}$ corresponds to the number of ordered factorisations of $n$ into integers greater than 1 (cf.\ OEIS A074206), which accords with the total number of perfect partitions of $n$ (see OEIS A002033).

Regarding $P(n, r)$ with $r \ge 2$, we note that its Dirichlet series (\ref{eq:DP}) coincides with that given in the Online Encyclopedia of Integer Sequences for the counts of ``ordered factorisations with increasing levels of parenthesis''
(e.g. see for parenthesis levels $1,2$ and $3$ respectively, the OEIS A050354, A050356 and A050358).
It is also remarked there that
the parenthesis level $r$ sequence $(a(n))_{n\in\mathbb N}$ satisfies the
recurrence relation
\[
a(n)=r\sum_{d\mid n,\ d<n}a(d)\quad (n\geq 2), \qquad a(1)=\frac{1}{r},
\]
which is consistent with Equation (\ref{115}) setting
$a(n) = P(n, r) + e(n)/r$,
and further that the partial sums of terms has asymptotic
\[
\sum_{n=1}^m \left (P(n,r)+e(n)\right )\asymp \frac{m^k}{r^2\, k\,\zeta^\prime(k)},
\]
where $k$ is such that $\zeta(k)=1+\frac{1}{k}$
(cf.\ the definition of $r_m$ in Theorem \ref{thm:DN} below).
\end{remark}

\section{Divisor function polynomials for squarefree integers}\label{sec:Bell}
For the statement of results in this section we first introduce some special sequences of numbers and related polynomials.

\begin{definition}
The {\it Bernoulli numbers\/} $B_n$ form a special sequence of rational numbers which are related to the Riemann zeta function at even integer arguments. The sequence begins $(B_0, B_1, B_2, B_3, B_4, B_5, B_6, \dots) = (1,-\frac{1}{2},\frac{1}{6},0,-\frac{1}{30},0,\frac{1}{42},\dots)$, defined by the generating function \cite[p.\ 48]{comtet}
\[
\frac{t}{e^t-1}=\sum_{n=0}^\infty \frac{t^n B_n}{n!};
\]
there is also the historic variant \cite[p.\ 107]{conway}
\[
\frac{t e^t}{e^t-1}=\sum_{n=0}^\infty \frac{t^n B_n^+}{n!},
\]
where $B_n^+=B_n$ for $n\neq 1$, $B_1^+=\frac{1}{2}$.

The {\it ordered Bell numbers\/} (or {\it Fubini numbers\/}) $F_n$,  and {\it ordered Bell polynomials\/} $F_n(x)$ are defined by
\[
F_n=\sum_{k=1}^nk!S(nk),\quad \text{and}\quad F_n(x)=\sum_{k=1}^nk!S(n,k)x^k,
\]
respectively,
where $S(n,k)$ are the Stirling numbers of the second kind. The ordered Bell numbers count the number of weak orderings or equivalently the number of distinct \emph{ordered} partitions of a set of $n$ elements \cite{tanny}.

Lastly, we define the {\it signed ordered Bell numbers\/} $F_n^-$ and {\it signed ordered Bell polynomials\/} $F_n^-(x)$ as
\[
F_n^-=\sum_{k=1}^n(-1)^{n-k}k!S(n,k)\quad \text{and}\quad F_n^-(x)=\sum_{k=1}^n(-1)^{n-k}k!S(n,k)x^k.
\]
\end{definition}

\smallskip\noindent
We now relate the divisor function polynomials of Section \ref{sec:dfp} specifically for squarefree numbers, to the
(standard and signed) Bell polynomials.

\begin{theorem}\label{thm115}
Let $n$ be a squarefree number with $\omega(n)=\Omega(n)$ distinct prime factors. Then
\[
x\,P(n,x)=F_{\omega(n)}(x),\qquad x\,Q(n,x)=F^-_{\omega(n)}(x).
\]
\end{theorem}
\begin{proof}
Every non-trivial divisor of the squarefree number $n$ is also squarefree, so the
number of non-trivial factorisations of $n$ is equal to the number of
non-trivial squarefree factorisations, i.e.\ $b_j(n) = c_j(n)$.
Moreover, factorising $n$ into $j$ non-trivial factors is equivalent to
assigning its $\Omega(n) = \omega(n)$ prime factors to $j$ non-empty sets,
with each non-trivial factor being the (commutative) product of the primes in
one of the sets. Hence
\begin{equation}\label{eq:sff}
c_j(n) = |c_j^{(-j)}(n)| = b_j(n) = S(\omega(n), j)\,j!.
\end{equation}
It follows that
\begin{equation*}
x\,P(n, x) = \sum_{j=1}^{\omega(n)} x^j\,c_j(n)
 = \sum_{j=1}^{\omega(n)} x^j\,j!\,S(\omega(n), j)
 = F_{\omega(n)}(x).
\end{equation*}
Similarly,
\begin{align*}
x\,Q(n, x) &= \sum_{j=1}^{\omega(n)} x^j\,c_j^{(-j)}(n)
 = \sum_{j=1}^{\omega(n)} x^j\,(-1)^{j+\omega(n)}\,b_j(n)
\\
 &= \sum_{j=1}^{\omega(n)} x^j\,(-1)^{j+\omega(n)}\,S(\omega(n), j)\,j!
 = F_{\omega(n)}^-(x).
\qedhere
\end{align*}
\end{proof}
\begin{remark}
We note that, by equations (\ref{enum22}) and (\ref{eq:sff}), the number of $m$-part sum systems when $n$ is squarefree is given by
\[
 {\mathcal N}_m(n) =
m!\sum_{L=0}^{\omega{(n)}}(-1)^{\omega(n)+L} \,L!\, S(L,m)\, S(\omega(n),L).
\]
\end{remark}
\medskip\noindent
Using a result of Tanny \cite{tanny}, Kellner observed that
$f_n(x) = F_n(x)/x$ has the symmetry
$f_n(x) = (-1)^{n-1}\,f_n(-1-x)$ \cite[Corollary 3.3]{kellner}, and by Theorem \ref{thm115} we can conclude
that
$P(n, x) = (-1)^{\omega(n)-1}\,P(n, -1-x)$.
Further, Theorem \ref{thm11} gives
\begin{align*}
 Q(n, x) &= P(n, x-1) % = (-1)^{\omega(n)-1} P(n, -1-(x-1))
 = (-1)^{\omega(n)-1}\,P(n, -x)
 = (-1)^{\omega(n)-1}\,Q(n, 1-x).
\end{align*}
Shifting the variable $x$ by $\pm\frac 12$, respectively, we hence obtain
the following symmetry statements.

\begin{corollary}
Let $n$ be a squarefree natural number. Then
\begin{align*}
P\left(n,x-\frac{1}{2}\right)&=(-1)^{\omega(n)-1}P\left(n,-x-\frac{1}{2}\right),
\\
 Q\left(n,x+\frac{1}{2}\right)&=(-1)^{\omega(n)-1}Q\left(n,-x+\frac{1}{2}\right) \qquad (x \in \mathbb C).
\end{align*}
In particular, the zeros of $P(n, \cdot)$ and the zeros of $Q(n, \cdot)$ are
point symmetric around $\frac 12$ and $-\frac 12$, respectively.
\end{corollary}

\begin{example}
The roots of $P(210, \cdot)$ and of $Q(210, \cdot)$ are
\begin{equation*}
x = -\frac 12 \pm \frac 12 \sqrt{\frac 12 \pm \frac 12 \sqrt{\frac 7{15}}}
\qquad\textrm{and}\qquad
x = \frac 12 \pm \frac 12 \sqrt{\frac 12 \pm \frac 12 \sqrt{\frac 7{15}}},
\end{equation*}
respectively.
(210 is the product of the first four primes.)
\end{example}
\medskip\noindent
We conclude this section with an observation on the connection between the
Bernoulli numbers and the non-trivial and balanced divisor functions.

\begin{theorem}
\label{thm11sf}
Let $n$ be a squarefree integer, $n \ge 2$. Then
\begin{equation*}
B_{\omega(n)} = \sum_{k=1}^{\omega(n)} \frac{(-1)^k\,c_k(n)}{k+1}
= \int_{-1}^0 x\,P(n, x)\,d x
= \int_{-1}^0 F_{\omega(n)}(x)\,d x
\end{equation*}
and
\begin{equation*}
B_{\omega(n)}^+ = \sum_{k=1}^{\omega(n)} \frac{c_k^{(-k)}(n)}{k+1}
= \int_0^1 x\,Q(n, x)\,d x
= \int_0^1 F_{\omega(n)}^-(x)\,d x.
\end{equation*}
%Let $m$ be a squarefree positive integer with $\omega(m)=t\geq 2$ distinct prime factors, so that $\Omega(m)=\omega(m)=t$,
%and
%$c_k=|c_k^{(-k)}|=k!S(t,k)$. Then
%with regard the Bernoulli numbers $B_n$, we have
%\[
%\mathop{B_t}_{t> 1}=\int_{0}^1 Q(m,x){\rm dx}
%=\sum_{k=0}^t \frac{(-1)^{t+k}k!S(t,k)}{k+1},
%\]
%so that the $t$-th Bernoulli number is given by
%\[
%B_t=\sum _{k=0}^{t}(-1)^{t+k} \frac{|c_k^{-k}(n)|}{k+1}=\sum _{k=0}^{t}(-1)^{t+k} \frac{b_k(n)}{k+1}
%\]
%\[
%B_t^-
%\sum_{k=0}^t \frac{c_k^{(-k)}(n)}{k+1}.
%\]
\end{theorem}
\begin{proof}
By \cite[Proposition 2.1]{kellner}
the Bernoulli numbers satisfy the identity
\begin{equation*}
B_j = \sum_{k=1}^j\frac{(-1)^k k!S(j,k)}{k+1} \qquad (j\geq 1).
\end{equation*}
Combining this formula with equation (\ref{eq:sff}),
we find
\begin{align*}
B_{\omega(n)} &= \sum_{k=1}^{\omega(n)} \frac{(-1)^k k!\,S(\omega(n), k)}{k+1}
= \sum_{k=1}^{\omega(n)} \frac{(-1)^k c_k(n)}{k+1}
\\
&= \sum_{k=1}^{\omega(n)} c_k(n) \int_{-1}^0 x^k\,d x
= \int_{-1}^0 \sum_{k=1}^{\omega(n)} x^k c_k(n)\,d x
= \int_{-1}^0 x\,P(n, x)\,d x.
\end{align*}
Similarly using equation (\ref{eq:sff}) and
$c_k^{(-k)}(n) = (-1)^{k+\omega(n)}\,|c_k^{(-k)}(n)|$, we also have
\begin{equation}\label{eq:Bwn}
B_{\omega(n)} = \sum_{k=1}^{\omega(n)} \frac{(-1)^k\,(-1)^{k+\omega(n)} c_k^{(-k)}(n)}{k+1}
= (-1)^{\omega(n)} \sum_{k=1}^{\omega(n)} \frac{c_k^{(-k)}(n)}{k+1}.
\end{equation}
Since $B_1 = - B_1^+$ and $B_j = 0$ for all odd $j \ge 3$ \cite[p.48 {[14d]}]{comtet}, it follows that
\begin{align*}
B_{\omega(n)}^+ &= \sum_{k=1}^{\omega(n)} \frac{c_k^{(-k)}(n)}{k+1}
\\
&= \sum_{k=1}^{\omega(n)} c_k^{(-k)}(n) \int_0^1 x^k\,d x
= \int_0^1 \sum_{k=1}^{\omega(n)} x^k\,c_k^{(-k)}(n)\,d x
= \int_0^1 x\,Q(n, x)\,d x.
\end{align*}
The equality with integrals over the (standard and signed) ordered Bell
polynomials follows by Theorem \ref{thm115}.
\end{proof}
\begin{remark}
The above theorem also shows the identities
\begin{equation*}
\sum_{k=1}^{\omega(n)} \frac{(-1)^k c_k(n)}{k+1} = 0,
\qquad
\sum_{k=1}^{\omega(n)} \frac {c_k^{(-k)}(n)}{k+1} = 0
\end{equation*}
whenever $n$ is the product of an odd number of distinct primes, but not a
prime itself.
\end{remark}
\medskip\noindent
In view of equation (\ref{eq:Bwn}) and the well-known identity of the
Riemann zeta function at even integer arguments \cite{mcl5}
\[
\zeta(2t)=\frac{(-1)^{t+1}2^{2t-1}\pi^{2t}B_{2t}}{2t!}.
\]
we obtain the following statement.
\begin{corollary}
For every integer $t \ge 1$ and every squarefree natural number $n$ such that $\omega(n) =\Omega(n)= 2 t$, we have
%For the Riemann zeta function at even integer arguments, and $n$ a squarefree integer with $\Omega(n)=\omega(n)=t$, we have that
\[
\zeta(2t)=\frac{(-1)^{t+1}2^{2t-1}\pi^{2t}}{2t!}\sum_{k=0}^{2t} \frac{c_k^{(-k)}(n)}{k+1}.
\]
\end{corollary}

\section{Dirichlet series of the sum system counting function}\label{sec:DN}
In this section we consider the Dirichlet series for the counting function $\mathcal{N}_m$ given in equation (\ref{enum22}).
We first observe that a different approach to the same counting problem
yields a different formula for the same number \cite{Hylock}, but their
equivalence follows from Theorem \ref{thm11}.
\begin{theorem}\label{thm:equiv}
Let $m\in\mathbb N$. Then
\begin{equation*}
\sum_{L=1}^{\infty} m!\,S(L, m)\,c_L^{(-L)} = \sum_{L=1}^{\infty} m!\,S(L-1, m-1)\,c_L.
\end{equation*}
\end{theorem}
\begin{remark}
The series in this theorem are pointwise finite sums, as $c_L^{(-L)}(n) = 0 = c_L(n)$ whenever $L > \Omega(n)$.
As a result of the theorem, we also have
\begin{equation}\label{eq:Hylock}
\mathcal N_m(n) = \sum_{L=1}^{\Omega(n)} m!\,S(L-1, m-1)\,c_L(n)
\end{equation}
for $n \ge 2$;
although superficially similar to equation (\ref{enum22}), this expression is
more convenient to work with as all terms in the sum are non-negative.
\end{remark}
\begin{proof}[Proof of Theorem \ref{thm:equiv}.]
Starting from the right-hand side and the identity
\begin{equation*}
m!\,S(L, m) = \sum_{k=0}^m (-1)^{m-k} {m\choose k} k^L
\end{equation*}
(see \cite[Theorem 5.1.A]{comtet}),
we find
\begin{align*}
\sum_{L=1}^\infty & m!\,S(L-1, m-1)\,c_L =
\sum_{L=1}^\infty  m \sum_{k=0}^{m-1} (-1)^{m-1-k} {m-1 \choose k} k^{L-1} c_L
\\
&= m \sum_{j=1}^m (-1)^{m-j} {m-1 \choose j-1} \sum_{L=1}^\infty (j-1)^{L-1} c_L
\\
&= m \sum_{j=1}^m (-1)^{m-j} {m-1 \choose j-1}\,P(\cdot, j-1)\\
&= m \sum_{j=1}^m (-1)^{m-j} {m-1 \choose j-1}\,Q(\cdot, j)
\\
&= m \sum_{j=1}^m (-1)^{m-j} {m-1 \choose j-1} \sum_{L=1}^\infty j^{L-1} c_L^{(-L)}
\\
&= \sum_{L=1}^\infty \sum_{j=1}^m (-1)^{m-j}\,\frac{(m-1)!}{(m-j)!\,(j-1)!}\,j^{L-1}\,c_L^{(-L)}
\\
&= \sum_{L=1}^\infty \sum_{j=0}^m (-1)^{m-j} {m\choose j} j^L c_L^{(-L)}
= \sum_{L=1}^\infty m!\,S(L, m)\,c_L^{(-L)}
\end{align*}
using Theorem \ref{thm11}.
\end{proof}
\begin{theorem}\label{thm:DN}
For natural numbers $m,n$, $m\ge 2$, the Dirichlet series for the arithmetic function $\mathcal{N}_m$, which counts the number of $m$-part sum systems, is given by
\begin{equation}\label{eq:DN}
\sum_{n=1}^\infty \frac{\mathcal{N}_m(n)}{n^s}
= \sum_{L=1}^\infty m!\,S(L-1, m-1)\,(\zeta(s) - 1)^L
\end{equation}
for $s\in\mathbb C$, $\Real s > 1$, such that %$|\zeta(s) - 1| < 1$;
$\zeta(\Real s) < \frac m{m-1}=1+\frac{1}{m-1}$.
The latter condition is satisfied whenever $\Real s > r_m$, with
the unique real number $r_m > 1$ such that $\zeta(r_m) = \frac m{m-1}$.

For $m=2$, we have $r_2 \approx 1.728649\dots$; for $m\ge 3$,
\begin{equation*}
\frac{\log (m-1)}{\log 2} < r_m < 1 + \frac{\log (m-1)}{\log\frac 3 2},
\end{equation*}
so $r_m$  is of order $\log m$ for large values of $m$.
\end{theorem}
\begin{proof}%[Proof of Theorem \ref{thm:DN}]
By equation (\ref{eq:Hylock}), we find
\begin{align*}
\sum_{n=1}^\infty \frac{\mathcal N_m(n)}{n^s} &= \sum_{n=1}^\infty \frac 1{n^s} \sum_{L=1}^{\Omega(n)} m!\,S(L-1, m-1)\,c_L(n)
\\
&= \lim_{N\rightarrow\infty} \sum_{n=1}^N \frac 1{n^s} \sum_{L=1}^{\hat\Omega(N)} m!\,S(L-1, m-1)\,c_L(n)
\\
&= \lim_{N\rightarrow\infty} \sum_{L=1}^{\hat\Omega(N)} m!\,S(L-1, m-1) \sum_{n=1}^N \frac{c_L(n)}{n^s}
\\
&= \lim_{N\rightarrow\infty} \sum_{L=1}^{\infty} m!\,S(L-1, m-1) \sum_{n=1}^N \frac{c_L(n)}{n^s}.
\end{align*}
The Stirling numbers of the second kind satisfy the estimate
$m!~\,S(L, m)~\le~m^L$
(cf.\ \cite[Theorem 5.2.C (3f)]{comtet}).
Therefore
\begin{align*}
\left| m!\,S(L-1, m-1)\sum_{n=1}^N \frac{c_L(n)}{n^s} \right|
&\le m!\,S(L-1, m-1) \sum_{n=1}^N \frac{c_L(n)}{|n^s|}
\\
&\le m\,(m-1)^{L-1} \sum_{n=1}^\infty \frac{c_L(n)}{n^{\Real s}}
\\
&= m\,(m-1)^{L-1}\,(\zeta(\Real s) - 1)^L,
\end{align*}
which are the ($N$-independent) terms of a convergent geometric series (with summation parameter $L$) if
$(m-1)\,(\zeta(\Real s) - 1) < 1$. This latter condition is easily seen to
be equivalent to
$\zeta(\Real s) < \frac m{m-1}$.
Furthermore, we have pointwise convergence (for fixed $L\in\mathbb N$)
\begin{equation*}
\lim_{N\rightarrow\infty} m!\,S(L-1, m-1) \sum_{n=1}^N \frac{c_L(n)}{n^s}
= m!\,S(L-1, m-1)\,(\zeta(s) - 1)^L,
\end{equation*}
so equation (\ref{eq:DN}) follows by dominated convergence.

It remains to see for which $s\in\mathbb C$ the convergence condition
$\zeta(\Real s) < \frac m{m-1}$ holds.
On the real interval $(1, \infty)$, the $\zeta$ function is strictly monotone
decreasing and continuous, with limits
$\zeta(r)\rightarrow\infty$ $(r \rightarrow 1)$ and
$\lim\limits_{r\rightarrow\infty} \zeta(r) = 1$.
Therefore there is a unique $r_m > 1$ such that $\zeta(r_m) = \frac m{m-1}$,
and we have convergence for all $s\in\mathbb C$ with $\Real s > r_m$.

If $m\ge 3$ and $r \ge 1 + \log(m-1)/\log\frac 3 2$, then
$r > 2$, so
\begin{equation*}
\log (m-1) \le (r - 1) \log \frac 3 2 < (r - 1) \log \frac 3 2 + \log (r - 1)
\end{equation*}
and therefore
\begin{equation*}
1 + \frac 1{r-1}\left(\frac 32\right)^{1-r} < 1 + \frac 1{m-1} = \frac m{m-1}.
\end{equation*}
The upper bound on $r_m$ in the theorem now follows by observing (using the
convexity of the function $x^{-r}$ $(x > 0)$) that
\begin{align*}
\zeta(r) &= 1 + \sum_{n=2}^\infty \frac 1{n^r}
 < 1 + \sum_{n=2}^\infty \int_{n-\frac 12}^{n+\frac 12} \frac 1{x^r}\,d x
= 1 + \int_{\frac 32}^\infty \frac{d x}{x^r}
= 1 + \frac 1{r-1} \left(\frac 32\right)^{1-r}.
\end{align*}
For the lower bound on $r_m$, we note that if $r \le \log(m-1)/\log 2$, then
$2^r \le m - 1$ and hence
\begin{equation*}
 \zeta(r) \ge 1 + \frac 1{2^r} \ge 1 + \frac 1{m - 1} = \frac m{m-1}
\qedhere
\end{equation*}
\end{proof}

\begin{remark}
At a casual glance, the first term in the formula
\cite[equation (1.7)]{ivic}
\[
\zeta(s)=\frac{s}{s-1}+s\int_1^\infty([t]-t)t^{-s-1}\, d t \qquad (s > 1)
\]
would suggest that $r_m \approx m$; however, the inequality in Theorem \ref{thm:DN} shows that the integral term causes a considerable adjustment.
A good estimate of the integral could lead to a better bound for $r_m$.
\end{remark}

\begin{remark}
Following the calculation in the proof of Theorem \ref{thm:DN} but starting
from equation (\ref{enum22}) instead of equation (\ref{eq:Hylock}) and using
the Dirichlet series for the balanced divisor function $c_j^{(-j)}$
(equation (\ref{eq:Dbal})),
we obtain an alternative form for the Dirichlet series of $\mathcal N_m$,
\begin{equation*}
\sum_{n=1}^\infty \frac{\mathcal N_m(n)}{n^s} = \sum_{L=1}^\infty m!\,S(L,m) \left(1 - \frac 1{\zeta(s)}\right)^L.
\end{equation*}
This series is absolutely convergent for $\Real s > r_m$, but we note that the
majorant
\begin{equation*}
\left| m!\,S(L, m) \sum_{n=1}^N \frac{c_L^{(-L)}(n)}{n^s} \right|
\le m!\,S(L, m) \sum_{n=1}^N \frac{b_L(n)}{n^{\Real s}}
\end{equation*}
\begin{equation*}
\le m!\,S\left(L,m)\,(\frac{\zeta(\Real s)}{\zeta(2 \Real s)} - 1\right)^L
\end{equation*}
(using $|c_L^{(-L)}| \le b_L$ and the Dirichlet series in Theorem \ref{thm:Db}) will be summable if and only
if $\frac{\zeta(s)}{\zeta(2s)} < \frac{m+1}m=1+\frac{1}{m}$, which is a slightly different
condition.
\end{remark}

\subsection*{Disclosure statement} No potential conflict of interests was reported by the authors.

\vspace{5mm}

\author{\noindent M. N. Huxley,\,\,\,M. C. Lettington,\,\,\,K. M. Schmidt \\
\,\,\,\,Department of Mathematics\\
Cardiff University,\,\,\,Abacws Building\\
Senghennydd Road,\,\,\, CF24 4AG\\
E-mail: Huxley@cardiff.ac.uk\\
\qquad LettingtonMC@cardiff.ac.uk\\
\qquad SchmidtKM@cardiff.ac.uk}

\end{document}